\documentclass[a4paper,12pt,twoside,notitlepage,reqno]{amsart}

\usepackage{amsmath,amssymb,amsbsy,amsthm,amsfonts}
\usepackage[hmargin=1.4in,vmargin=1.4in]{geometry}
\usepackage{url}

\usepackage{tikz}
\usepackage{polynom}
\usetikzlibrary{matrix}
\usepackage[linktocpage]{hyperref}
\hypersetup{
    colorlinks=true,        
    linkcolor=red,          
    citecolor=red,         
    filecolor=magenta,      
    urlcolor=cyan           
}

\addtocontents{toc}{\protect\setcounter{tocdepth}{1}}

\usepackage{eucal}

\theoremstyle{plain}
	\newtheorem{theorem}{Theorem}
	
		\numberwithin{theorem}{section}
	\newtheorem{lemma}[theorem]{Lemma}
	\newtheorem{prop}[theorem]{Proposition}
	\newtheorem{corollary}[theorem]{Corollary}

	\newtheorem*{theorem*}{Theorem}
	\newtheorem*{lemma*}{Lemma}
	\newtheorem*{prop*}{Proposition}
	\newtheorem*{cor*}{Corollary}
	\newtheorem*{conj*}{Conjecture}

\theoremstyle{definition}
	
	\newtheorem*{example*}{Example}

\begin{document}

\title[]
{Filtered deformations of commutative algebras of Krull dimension two}

\author{Jason P. Bell}
\address{University of Waterloo \\
Department of Pure Mathematics \\
Waterloo, Ontario \\
Canada  N2L 3G1}
\thanks{The author was supported in part by NSERC grant RGPIN-2022-02951.}
\email{jpbell@uwaterloo.ca}
\begin{abstract}
    Let $F$ be an algebraically closed field of positive characteristic and let $R$ be a finitely generated $F$-algebra with a filtration with the property that the associated graded ring of $R$ is an integral domain of Krull dimension two.  We show that under these conditions $R$ satisfies a polynomial identity, answering a question of Etingof in the affirmative in a special case.
\end{abstract}
\maketitle

\section{Introduction}
The field of quantum algebras typically involves the study of algebras that are constructed by deforming the multiplicative structure of a ``classical'' object (often the coordinate ring of an affine variety, an enveloping algebra, or a group algebra). An important class of such deformations comes from algebras that are filtered and whose associated graded algebra with respect to this filtration is commutative.  Given a field $F$ and a graded $F$-algebra $B$, a $\mathbb{Z}_+$-\emph{filtered deformation} (henceforth we shall drop the $\mathbb{Z}_+$) of $B$ is simply an algebra $R$ with a filtration 
\begin{equation}\label{eq:filt}
V_0\subseteq V_1 \subseteq V_2 \subseteq \cdots \subseteq \bigcup_i V_i=R
\end{equation} such that $B$ is the associated graded algebra of $R$.  An important class of filtered deformations comes from enveloping algebras and their homomorphic images, which are filtered deformations of commutative algebras.  Indeed, if the graded algebra $B$ above is commutative and generated in degree one then $R$ is necessarily a homomorphic image of an enveloping algebra of a Lie algebra.  

In positive characteristic, enveloping algebras are known to satisfy a polynomial identity.  Intuitively, this says that filtered deformations of commutative algebras that are generated in degree one behave much like algebras that are finite modules over their centres.  Etingof \cite[Question 1.1]{Etingof} asked whether this phenomenon holds more generally for all filtered deformations of graded commutative algebras over positive characteristic algebraically closed fields and he proved that when a filtered deformation of a commutative algebra in positive characteristic satisfies a polynomial identity then the degree of the identity is a power of the characteristic of the field.

Understanding when filtered deformations in positive characteristic satisfy a polynomial identity is often useful even in the context of looking at filtered deformations of algebras in characteristic zero.  The reason for this is that one is typically dealing with finitely presented algebras, which allows one to take a model for the algebra over a base ring that is finitely generated as a $\mathbb{Z}$-algebra.  One can then reduce mod maximal ideals of the base ring to obtain filtered deformations of algebras in positive characteristic. In some circumstances this approach can be used to lift mod $p$ information to prove non-trivial results in characteristic zero (see for example \cite{BD}), and this philosophy inspired in part a related question of Cuadra, Etingof, and Walton \cite{CEW}.   

In general, Etingof's question has proved difficult to answer, although it has been answered in the affirmative in the case of iterated Hopf Ore extensions in positive characteristic \cite{BZ}.  We are able to prove Etingof's question has an affirmative answer in the case when the associated graded algebra is finitely generated and has Krull dimension at most two.  In fact, the case of Krull dimension one follows from a result of Small and Warfield \cite{SW, SSW} and so the interesting case of this result is the Krull dimension two case.
\begin{theorem}
\label{thm:main} Let $F$ be an algebraically closed field of positive characteristic and let $R$ be a finitely generated filtered $F$-algebra whose associated graded ring is a commutative $F$-algebra of Krull dimension at most two.  Then $R$ satisfies a polynomial identity.
\end{theorem}
One of the extra layers of difficulty when considering the case of algebras that are not generated in degree one is that for elements of $R$ that are not of degree one (that is, not in the subspace $V_1$ of the filtration), the adjoint operators ${\rm ad}_x=[x,-]: R\to R$ do not tend to preserve generating subspaces for the algebra $R$, which is a useful property to exploit when proving that an algebra satisfies a polynomial identity.  

We are able to get around this added layer of difficulty in the Krull dimension two case as follows.  We note that for the filtration in (\ref{eq:filt}), the $F$-dimension of $V_n$ is growing quadratically with $n$ and one can use this to show that if $x\in R\setminus F$ then the dimension of $F(x)V_n$ as a left $K$-vector space is growing linearly with $n$ and we in fact have $F(x)V_n/F(x)V_{n-1}$ has uniformly bounded dimension (see Lemma \ref{lem:O1}).

We use this observation to show that for suitably chosen $x\in R$ and $n\in \mathbb{N}$ the space of elements mapped into $F(x)V_n$ under the adjoint operator ${\rm ad}_x$ has large dimension as an $F(x)$-vector space and from this we deduce that the centralizer of $x$ in the quotient division algebra of $R$ is infinite-dimensional over $K$.  From here we argue that the centralizer contains a subfield of transcendence degree at least two over $F$ and use this fact to show that $R$ is PI.


\section*{Proof of Theorem \ref{thm:main}} 
In this section we prove our main result. We begin with a basic result in commutative algebra, which is surely known, although we're unaware of a reference that does everything claimed.

\begin{prop} Let $F$ be an algebraically closed field of characteristic $p$ and let $B$ be a finitely generated commutative connected $\mathbb{N}$-graded $F$-algebra that is an integral domain of Krull dimension $d$ whose degree $m$ piece is nonzero for all sufficiently large $m$. Then there exist $x_1,\ldots ,x_d$ in $B$ such that the following hold:
\begin{enumerate} 
\item[(a)] $x_1,\ldots, x_d$ are homogeneous of degree $n$ with $\gcd(p,n)=1$;
\item[(b)] $\{x_1,\ldots ,x_d\}$ is a separating transcendence base for ${\rm Frac}(B)$ as an extension of $F$;
\item[(c)] $B/x_1B$ is reduced.

\end{enumerate}
\label{prop:reduced}
\end{prop}
\begin{proof}
If we invert the nonzero homogeneous elements of $B$ we obtain a graded Laurent polynomial algebra ${\rm Frac}_{\rm gr}(B)=K[t^{\pm 1}]$ where $K$ is a finitely generated extension of $F$ of transcendence degree $d-1$ formed by taking the degree zero piece of the localization and $t$ is homogeneous of degree one.  Then $K$ has a separating transcendence base $\{a_1/b,\ldots ,a_{d-1}/b\}$ with $a_1,\ldots ,a_{d-1},b\in B$ nonzero homogeneous elements of the same degree \cite[Theorem 9.27]{Milne}.  By multiplying $a_1,\ldots ,a_{d-1},b$ by a common homogeneous factor if necessary, we may assume that they have the same degree $n$ with $n$ and $p$ coprime.  It follows that if we let $C$ denote the graded subalgebra of $B$ generated by elements of degree $n$, then since $F$ is infinite, by taking a basis for the elements of degree $n$ as a generating set for $C$ we see by Noether normalization \cite[Ch. V, \S4, Theorem 8]{ZSvol1}
 that there exist $x_1,\ldots ,x_n\in B_n$ such that $C$ is integral over $F[x_1,\ldots ,x_n]$.  Moreover, since $F$ is algebraically closed, this version of Noether normalization shows we may take $\{x_1,\ldots ,x_n\}$ to be a separating transcendence base for ${\rm Frac}(C)$ over $F$.  Then by construction ${\rm Frac}(C)\supseteq F(a_1/b,\ldots ,a_{d-1}/b)(t^n)$ and so ${\rm Frac}(B)=K(t)$ is separable over ${\rm Frac}(C)$ and thus also over $F(x_1,\ldots ,x_n)$.  
 
 If we invert the nonzero homogeneous elements of $F[x_1,\ldots ,x_n]$ we obtain a PID such that $F(a_1/b,\ldots ,a_{d-1}/b)[u^{\pm 1}]$ with $u$ homogeneous of degree $n$ if a finite free module over this localization.  Hence ${\rm Frac}_{{\rm gr}}(B)=K[t^{\pm 1}]$ is a finite free module over ${\rm Frac}_{\rm gr}(F[x_1,\ldots ,x_n])$.  Thus by a graded version of Grothendieck's generic freeness lemma, there exists a nonzero homogeneous element $f\in F[x_1,\ldots ,x_n]$ such that 
$B_f$ is a finitely generated free $F[x_1,\ldots ,x_n]_f$-module.  

In particular we have
$$B_f = F[x_1,\ldots ,x_n]_f h_1\oplus \cdots \oplus F[x_1,\ldots ,x_n]_f h_s$$ for some $h_1,\ldots ,h_s\in B$ and we can write
$h_i h_j = \sum b_{i,j,\ell} h_{\ell}$ with $b_{i,j,\ell}\in F[x_1,\ldots ,x_n]_f$ uniquely determined.

Now since $f\in F[x_1,\ldots ,x_n]$ is nonzero and homogeneous, it is nonzero on a Zariski open subset $\Omega$ of $F^n$.  

Consider $\overline{F(x_1,\ldots ,x_n)}[t_1,\ldots ,t_s]$.  
Then $$A:={\rm Frac}(B_f)\otimes_{F(x_1,\ldots,x_n)} 
\overline{F(x_1,\ldots,x_n)}$$ is a finite free module over $
\overline{F(x_1,\ldots,x_n)}$ with basis $$h_1\otimes 1,\ldots , 
h_s\otimes 1$$ and it is reduced as ${\rm Frac}(B)$ is a 
separable extension of $F(x_1,\ldots,x_n)$.  

It follows that the ideal generated by
$$\sum_{i,j} b_{i,j,\ell} t_i t_j$$ in $\overline{F(x_1,\ldots,x_n)}[t_1,\ldots ,t_s]$ contains a power of $(t_1,\ldots,t_s)$, because if it did not, there would be a nonzero solution $(\alpha_1,\ldots,\alpha_s)\in \overline{F(x_1,\ldots,x_n)}^s$ to the system of equations
$$\sum_{i,j} b_{i,j,\ell} \alpha_i \alpha_j=0$$ and this would then say $$\left(\sum h_i\otimes \alpha_i \right)^2 = 0$$ in $A$.  But since $A$ is reduced and the $h_i$'s form a basis for $A$ over 
$\overline{F(x_1,\ldots,x_n)}$ we then must have $\alpha_1=\cdots =\alpha_s=0$, a contradiction.  It follows that there is some $N$ such that for $i=1,\ldots ,s$ we have
$$t_i^N  = \sum_{\ell} q_{\ell}(t_1,\ldots ,t_s)\left(\sum_{i,j} 
b_{i,j,\ell} t_i t_j\right).$$  Then the coefficients of the 
$q_{\ell}$ lie in a finite extension $L$ of $F(x_1,\ldots,x_n)$.  
We may pick a basis $1=\epsilon_0,\ldots, \epsilon_m$ for 
$L$ as a $F(x_1,\ldots,x_n)$-vector space and we may write 
each $q_{\ell}$ as a sum $\sum_{i=0}^m q_{\ell,i}\epsilon_i$ 
where each $q_{\ell,i}$ is an element of $F(x_1,\ldots,x_n)[t_1,\ldots, t_s]$.  
Since each $b_{i,j,\ell}\in F[x_1,\ldots,x_n]$, we see if apply 
the $F(x_1,\ldots,x_n)$-linear map $T: L\to F(x_1,\ldots,x_n)
$ which sends $\epsilon_i$ to $\delta_{i,0}$ then after 
clearing denominators we get a relation 
$$c(x_1,\ldots ,x_n)t_k^N \in \sum_{\ell} F[x_1,\ldots ,x_n]
\left(\sum_{i,j} b_{i,j,\ell} t_i t_j\right),$$ with 
$c(x_1,\ldots ,x_n)\neq 0$, which holds for all 
$k=1,\ldots ,s$.  

Then if we choose $(\lambda_1,\ldots ,\lambda_n)\in F^n\setminus \{(0,\ldots ,0)\}$ so that $c(x_1,\ldots ,x_n), f\not\in (\lambda_1 x_1+\cdots+ \lambda_n x_n)F[x_1,\ldots ,x_n]$ then we see that if $a_1,\ldots ,a_s\in F[x_1,\ldots,x_n]$ are such that $\left(\sum a_i h_i\right)^2 \in (\lambda_1 x_1+\cdots+ \lambda_n x_n)B_f$ then there is some $M$ such that $$f^M \sum a_i a_j b_{i,j,\ell} \in (\lambda_1 x_1+\cdots+ \lambda_n x_n)F[x_1,\ldots ,x_n]$$ for all $\ell$ and so 
$c(x_1,\ldots ,x_n) f^M a_e^N \in (\lambda_1 x_1+\cdots+ \lambda_n x_n)F[x_1,\ldots ,x_n]$ and so $a_e\in (\lambda_1 x_1+\cdots+ \lambda_n x_n)$ for all $e$ by our choice of $(\lambda_1,\ldots ,\lambda_n)\in F^n$  

Thus $B/ (\lambda_1 x_1+\cdots+ \lambda_n x_n)$ is reduced for all $(\lambda_1,\ldots ,\lambda_n)\in F^n\setminus \{(0,\ldots ,0)\}$ outside of a Zariski closed subset of $F^n$ and by making a linear change of variables, we may assume that $B/x_1B$ is reduced.
\end{proof}

For the remainder of this section we make the following assumptions and adopt the following notation.
\begin{enumerate}
\item We let $F$ be an algebraically closed field of positive characteristic $p>0$. 
\item We let $R$ be a finitely generated $F$-algebra of Gelfand-Kirillov dimension two that has a filtration 
$$F=V_0\subseteq V_1\subseteq V_2\subseteq \cdots $$ with $\bigcup V_i = R$ such that the associated graded ring 
$$B:=F\oplus \left(\bigoplus_{i\ge 1} V_i/V_{i-1}\right)$$ is a commutative domain such that $B_n$ is nonzero for all $n$ large.  
\item For each nonzero $a\in R$ there is a unique smallest $n$ for which $a\in V_n$ and we call this the \emph{degree} of $a$ and write $d(a)$ for this quantity.  
\item For a nonzero element $a$ of degree $n$ we let $\overline{a}$ denote the image of $a$ in $V_n/V_{n-1}$.  
\item Since $R$ is noetherian domain, it has a division algebra of quotients by Goldie's theorem, and we let $D={\rm Frac}(R)$.  
\item We let $x$ be an element of $R$ such that $B/\overline{x}B$ is reduced, and we let $d$ denote the degree of $x$.  Such an $x$ exists by Proposition \ref{prop:reduced}.
\item We let $K\subseteq D$ denote the fraction field of $F[x]\subseteq R$ and we let $E$ denote the centralizer of $x$ in $D$.  
\item Inside of $D$ we can form the left $K$-vector spaces $KV_n$ formed by taking the left $K$-span of elements of $V_n$, and we let $\Omega_n\subseteq KR=\bigcup_i KV_i$ denote the set of elements $a$ such that $[a,x]\in KV_n$. 
\end{enumerate}
We note that assumption (2) is in general not implicit when we consider filtered deformations.  Nevertheless we shall show later on that it is of no loss of generality to make this assumption.

We begin with a lemma that shows that when we work with $K$-vector spaces, the dimension of $KV_n$ grows linearly with $n$.
\begin{lemma} 
\label{lem:O1}
We have ${\rm dim}_K(KV_n/KV_{n-1}) = O(1)$.  
\end{lemma}
\begin{proof} Since the algebra $B/\overline{x}B$ is a finitely generated commutative $F$-algebra of Krull dimension one, its Hilbert series is rational with all poles at roots of unity and all poles simple. Consequently, the Hilbert series has eventually periodic coefficients.  It follows that there is some $\kappa$ such that for all $n$, $\left(B/\overline{x}B\right)_n$ has dimension at most $\kappa$.
Let $n\ge 1$.  Then there exist $s\le \kappa$ and $a_1,\ldots ,a_s\in V_n$ such that their images in
$(B/\overline{x}B)_n$ span this space as a $F$-vector space.  We claim that $KV_n = KV_{n-1} + Ka_1+\cdots + Ka_s$.  To see this, let $a\in V_n$.  Then there are $\lambda_1,\ldots ,\lambda_s\in F$ such that $a-\sum_{i=1}^s \lambda_i a_i \in V_{n-1} + xV_{n-d}$.  In particular,
$a\in KV_{n-1} + Ka_1+\cdots + Ka_s$.  Since the elements $a\in V_n$ span $KV_n$ as a left $K$-vector space, we obtain that $KV_n/KV_{n-1}$ has dimension $\le \kappa$ for all $n$.  The result follows.
\end{proof}
We need the following lemma, which will give the estimates we will need in the proof of Theorem \ref{thm:main}.

\begin{lemma} Let $r, d$, and $N$ be positive integers. Then there exist positive integers $L=L(d)$ and $M=M(r,d,N)$ with the following properties:
\begin{enumerate}
\item[(a)] for $i=L+1,\ldots ,L+r$ we have $(\lceil M/p^i\rceil -1)p^i + d < p^L M$;
\item[(b)] the integers $\lceil M/p^i\rceil p^i$ for $i=L+1,\ldots ,L+r$ are pairwise distinct and larger than $p^{L+r} N$ and are strictly increasing with $i$.
\end{enumerate}
\label{lem:M}
\end{lemma}
\begin{proof} Pick $L$ such that $p^L>d$ and pick $M' > N+1$ and let 
$$M=p^{L+r+1} M' - p^{L+1}- p^{L+2}-\cdots - p^{L+r}.$$
Then for $i\in \{L+1,\ldots ,L+r\}$, 
$$\lceil M/p^i \rceil = p^{L+r+1-i} M' - 1 - p - \cdots - p^{L+r-i}$$
and so 
\begin{align*}
(\lceil M/p^i\rceil -1)p^i + d &=p^{L+r+1} M' - 2p^i - p^{i+1} - \cdots - p^{L+r}+d\\
&= M - p^i + \left(\sum_{j=L+1}^{i-1} p^j \right) + d < M, \end{align*}
where the last step follows from the fact that $d<p^{L}$ and the fact that 
$$p^{L+1} +\cdots + p^{i-1} < p^i-p^L.$$
This establishes (a).  The above also shows that 
$\lceil M/p^i\rceil p^i= p^{L+r+1}M' - p^i - \cdots - p^{L+r}$, and so we see the integers are pairwise distinct and increasing with $i$.  The fact that these numbers are all larger than $p^{L+r} N$ follows from our choice of $M'$ and the fact that $p^{L+r+1}M' - p^i - \cdots - p^{L+r}> p^{L+r+1} (M'-1)$.
\end{proof}
\begin{prop}
We have $\limsup_n \left({\rm dim}_K(\Omega_n) - {\rm dim}_K(KV_n) \right) =\infty$.
\label{prop:infinity}
\end{prop}
\begin{proof}
 Let $r$ be an arbitrary positive integer.  We apply Lemma \ref{lem:M} with this $r$, $d=d(x)$ (the degree of $x$), and $N$ such that $B_n/xB_{n-d}$ has positive dimension as an $F$-vector space for all $n\ge N$.  
 
Then consider the set $U_M$ of elements $a\in R$ such that $[a,x]\in V_M$.  Then $U_M\subseteq \Omega_M$. For $i\in \{L+1,\ldots ,L+r\}$, we let $\ell_i=\lceil M/p^i\rceil$.  Then $\ell_i>N$ by part (b) of Lemma \ref{lem:M}, and so for $i\in \{L+1,\ldots ,L+r\}$ there is an element $a_i$ whose image in $V_{\ell_i}/(xV_{\ell_i-d} + V_{\ell_i-1})$ is nonzero.  Then since $B/\overline{x}B$ is reduced, we see that 
the image of $a_{i}^{p^i}$ is nonzero in $V_{p^i \ell_i}/(xV_{p^i \ell_i-d} + V_{p^i \ell_i-1}).$  Then since $d({\rm ad}_v(w)) < d(v)+d(w)$ for all $v,w\in R$ and since ${\rm ad}_{v^{p^i}} = {\rm ad}_v^{p^i}$, we see by Lemma \ref{lem:M} that
$[a_{i}^{p^i},x]$ has degree at most 
$$p^i(\ell_i-1)+ d < p^i(\lceil M/p^i\rceil -1)+d<M.$$
 
Hence $a_{i}^{p^i}\in U_M$.  We also have by Lemma \ref{lem:M} that $p^{L+1} \ell_{L+1},\ldots , p^{L+r}\ell_{L+r}$ are pairwise distinct.  
It follows that the $a_{L+1}^{p^{L+1}},\ldots ,a_{L+r}^{p^{L+r}}$ are linearly independent mod $xR$ since they have distinct degrees and the image of $a_i^{p^i}$ has nonzero image in the degree $\ell_i p^i$ part of $B/\overline{x}B$ since this ring is reduced.
  
Thus $V_{M-d+1} + \sum_{i=L+1}^{L+r} F a_i^{p^i} \subseteq U_M$.  We now let $u_1,\ldots ,u_e$ be a $K$-basis for $KV_n$ with $d(u_1)\ge \cdots \ge d(u_e)$ and with the additional property that if 
$u_1',\ldots ,u_e'$ is another $K$-basis for $KV_n$ with 
$d(u_1')\ge \cdots \ge d(u_e')$ then either $d(u_i)=d(u_i')$ for $i=1,\ldots, e$ or there is some $j$ such that $d(u_j') > d(u_j)$ and $d(u_i)=d(u_i')$ for $i<j$. 
 
We now claim that $$\{u_1,\ldots, u_e\}\cup \{a_i^{p^i}\colon i=L+1,\ldots ,L+r\}$$ is a left $K$-linearly independent subset of $\Omega_M$.  The fact that these elements are in $\Omega_M$ has been established, so it suffices to prove independence.  

To see independence, suppose that we have some nontrivial dependence
$$\sum_{i=1}^e \lambda_i u_i + \sum_{j=L+1}^{L+r} \gamma_j a_{j}^{p^j} =0$$ with $\gamma_{i},\lambda_j\in K$.  Then after clearing denominators, we may assume that the $\lambda_i$ and $\gamma_{j}$ are in $F[x]$.  Then if we consider the image in the associated graded ring, we see that we have a non-trivial homogenous relation of the form
$$\sum_{i=1}^e \alpha_i \overline{x}^{m_i} \overline{u_i} + \sum_{i=L+1}^{L+r} \beta_{i} \overline{x}^{n_{i}} \overline{a_{i}}^{p^i} = 0$$ with the $\alpha_i,\beta_{i}\in F$ and $m_i,n_{i}$ nonnegative integers such that there is some $N_0$ such that 
$N_0 = m_{\ell} d + d(u_{\ell}) = n_{i}d +p^i d(a_{i})$ for all $\ell$ and all $i$; moreover, since the associated graded ring is an integral domain we may assume that the minimum of the $m_i$ and $n_{i}$ is zero.  

There are now two cases to consider.  The first case is when some $\beta_i\neq 0$.  Then we let $i_0$ be the largest index $i$ for which some $\beta_{i}\neq 0$.  Then by Lemma \ref{lem:M}, $p^{i_0} d(a_{i_0}) > p^i d(a_{i})$ for $i<i_0$ and since $p^{i_0} d(a_{i_0})> n \ge d(u_i)$ for all $i$, we then see that $n_{i_0}=0$ and $m_i>0$ for all $i$ and $n_{i}>0$ for $i<i_0$ and so our homogeneous relation implies
$$\overline{a_{i_0}}^{p^{i_0}} \in \overline{x}B.$$
Since the ideal $\overline{x}B$ is reduced, this then implies $\overline{a_{i_0}}$
is in $\overline{x}B$, which contradicts our manner of choosing this element.  

Thus we may assume that $\beta_j=0$ for all $j$ and so our non-trivial relation is of the form
$$\sum_{i=1}^e \alpha_i \overline{x}^{m_i} \overline{u_i}=0.$$
We now let $q$ denote the smallest index $i$ for which $\alpha_i\neq 0$. Then since $d(u_1)\ge \cdots \ge d(u_e)$, we see that $m_q=0$ and after scaling we may assume that $\alpha_q=1$ and so our homogeneous relation is of the form
$$\overline{u_q} + \sum_{i=q+1}^e \alpha_i \overline{x}^{m_i} \overline{u_i} = 0.$$ 

In particular, if we let $u_i' = u_i$ for $i\neq q$ and let $$u_q' = u_q + \sum_{i=q+1}^e c_i x^{m_i} u_i,$$ then $d(u_q')<d(u_q)$ and 
$u_1',\ldots, u_e'$ are a $K$-basis for $KV_n$.  After reindexing to ensure that the degrees of the $u_i'$ are weakly decreasing with $i$, we then see that the smallest index $j$ for which $d(u_j')\neq d(u_j)$ necessarily satisfies $d(u_j') < d(u_j)$, which contradicts how the basis $u_1,\ldots ,u_e$ was chosen.  Thus
$\{u_1,\ldots, u_e\}\cup \{a_i^{p^i}\colon i=L+1,\ldots ,L+r\}$ is $K$-linearly independent.  So we see ${\rm dim}_K(\Omega_M) \ge {\rm dim}(KV_{M-d}) + r$.  By Lemma \ref{lem:O1} we have 
${\rm dim}_K(KV_n) - {\rm dim}_K(V_{n-d}) = O(1)$ and so 
${\rm dim}_K(\Omega_M) - {\rm dim}_K(KV_M) \ge r - O(1)$, where the $O(1)$ constant depends on $d$ but not on $r$ or $M$.  Since $r$ is arbitrary, we then see that the supremum of ${\rm dim}_K(\Omega_n) - {\rm dim}_K(KV_n)$ over $n$ is infinite. The result follows.
\end{proof}
\begin{corollary} We have ${\rm dim}_K(E)=\infty$.
\label{cor:E}
\end{corollary}
\begin{proof} We have a $K$-linear map
$$\Omega_n \to KV_n$$ given by $a\mapsto [a,x]$.  Then the kernel of this map is $E\cap \Omega_n$ and so the rank-nullity theorem gives the inequality
$${\rm dim}_K(\Omega_n) \le {\rm dim}_K(KV_n) + {\rm dim}_K(E).$$  In particular, if $E$ is finite-dimensional as a $K$-vector space then ${\rm dim}_K(\Omega_n) - {\rm dim}_K(KV_n)$ is uniformly bounded, which contradicts the conclusion of Proposition \ref{prop:infinity}.  It follows that $E$ is infinite dimensional as a $K$-vector space.
\end{proof}

We are now ready to prove our main result.

\begin{proof}[Proof of Theorem \ref{thm:main}]
We note that we may make the assumption that the associated graded ring of $R$ is connected, because if it is not, $V_0$, the degree zero piece of the filtration, is a commutative subalgebra of $R$ that properly contains $F$ and hence it has an element $y$ that is transcendental over $F$.  Then ${\rm ad}_y$ is a locally nilpotent derivation of $R$ and since $R$ is finitely generated $y^{p^n}$ is central for some $n$.  But then $R$ has Gelfand-Kirillov dimension two and has a centre of Gelfand-Kirillov dimension at least one and hence $R$ satisfies a polynomial identity by a result of Smith and Zhang \cite{SZ}.

Thus we may assume that the associated graded ring of $R$ is connected and so we now adopt the assumptions and notation from items (1)--(8) earlier in this section.  Then by Corollary \ref{cor:E} we have that if $E$ is the centralizer of $x$ in $D={\rm Frac}(R)$ then $E$ is an infinite-dimensional $K$-vector space.

We claim there is some $a\in E$ that is not algebraic over $K$.  To see this, suppose towards a contradiction that this is not the case.  Then every $a\in E$ has the property that $K(a)$ is an algebraic extension of $K$. 
Then since $F$ is algebraically closed and $B$ has a separating transcendence base over $F$, we have that $B\otimes_F L$ is a noetherian integral domain for every field extension $L$ of $F$ and so $R\otimes_F L$ is a noetherian domain for every field extension $L$ of $F$.  In particular if we take $L$ to be a maximal subfield of $E$ then $R\otimes_F L$ is noetherian and hence localization gives that $D\otimes_F L$ is noetherian; finally since $D\otimes_F L$ is a free $L\otimes_F L$ module, faithful flatness shows that $L\otimes_F L$ is noetherian and so a result of V\'amos \cite{Vam} gives that $L$ is finitely generated as an extension of $F$ and since $L$ is algebraic over $K$ we also have that $[L:K]<\infty$ and so $[L:Z(E)]<\infty$.   Since $L$ is a maximal subfield of $E$, it is equal to its own centralizer in $E$ and so the fact that ${\rm dim}_K(L)<\infty$ gives that ${\rm dim}_L(E)<\infty$ \cite[Theorem 15.4]{Lam}.  But this then gives that $E$ is finite-dimensional over $K$, a contradiction.

It follows that there is some element $a\in E$ such that $K(a)$ has transcendence degree one over $K$ then $F(x,a)$ is a field extension of Gelfand-Kirillov dimension two and so \cite[Theorem 1.4]{BellDivision} gives that $D$ is finite-dimensional over its centre and hence $R$ satisfies a polynomial identity.
\end{proof}
\section*{Acknowledgments}
We thank Pavel Etingof for helpful discussions.

\end{document}